\documentclass[12pt]{amsart}
\usepackage{amssymb}
\usepackage{amsthm}
\usepackage{enumerate}
\usepackage[dvips]{graphics,color}

\theoremstyle{plain}
\newtheorem{Thm}{Theorem}[section]

\newtheorem{Cor}[Thm]{Corollary}

\theoremstyle{remark}

\small\normalsize
\numberwithin{equation}{section}

\newcommand{\T}{{\mathbb T}}

\newcommand{\R}{{\mathbb R}}
\newcommand{\D}{{\mathbb D}}

\begin{document}

\title{Exceptional sets for the derivatives of Blaschke products}

\address{Universit\'e de Lyon, Lyon, F-69003, France; Universit\'e Lyon I, Institut Camille Jordan,
 Villeurbanne Cedex F-69622, France; CNRS, UMR5208, Villeurbanne,
 F-69622, France.}
\email{fricain@math.univ-lyon1.fr}

\author{Emmanuel Fricain, Javad Mashreghi}
\address{D\'epartement de math\'ematiques et de statistique,
         Universit\'e Laval,
         Qu\'ebec, QC,
         Canada G1K 7P4.}
\email{Javad.Mashreghi@mat.ulaval.ca}

\thanks{This work was supported by NSERC (Canada) and FQRNT (Qu\'ebec)}



\begin{abstract}
We obtain growth estimates for the logarithmic derivative
$B'(z)/B(z)$ of a Blaschke product as $|z| \to 1$ and $z$ avoids
some exceptional sets.
\end{abstract}

\maketitle

\section{Introduction}

Let $f$ be a meromorphic function in the unit disc $\D$. Then its
order is defined by
\[
\sigma = \limsup_{r \to 1^-} \frac{\log^+ T(r)}{\log1/(1-r)},
\]
where
\[
T(r) = \frac{1}{\pi} \,\, \int_{\{|z|<r\}} \,\,
\frac{|f'(z)|^2}{(1+|f(z)|^2)^2} \,\, \log(\, \frac{r}{|z|} \,) \,\,
dx \, dy
\]
is the Nevanlinna characteristic of $f$ \cite{nev70}. Meromorphic
functions of finite order have been extensively studied and they
have numerous applications in pure and applied mathematics, e.g. in
linear differential equations. In many applications a major role is
played by the logarithmic derivative of meromorphic functions and we
need to obtain sharp estimates for the logarithmic derivative as we
approach to the boundary \cite{gun88, gun98}.   In particular, the
following result for the rate of growth of meromorphic functions of
finite order in the unit disc has application in the study of linear
differential equations \cite[Theorem 5.1]{heit00}.

\begin{Thm}  \label{T:CGH}
Let $f$ be a meromorphic function in the unit disc $\D$ of finite
order $\sigma$ and let $\varepsilon>0$. Then the following two
statements hold.
\begin{enumerate}[(a)]
\item There exists a set $E_1 \subset (0,1)$ which satisfies
\[
\int_{E_1} \frac{dr}{1-r} < \infty,
\]
such that, for all $z \in \D$ with $|z| \not\in E_1$, we have
\begin{equation} \label{f'/f}
\bigg|\, \frac{f'(z)}{f(z)} \, \bigg| \leq
\frac{1}{(1-|z|)^{3\sigma+4+\varepsilon}}.
\end{equation}

\item There exists a set $E_2 \subset [0,2\pi)$ whose Lebesgue
measure is zero and a function $R(\theta) : [0,2\pi) \setminus E_2
\longrightarrow (0,1)$ such that for all $z=re^{i\theta}$ with
$\theta \in [0,2\pi) \setminus E_2$ and $R(\theta)<r <1$ the
inequality (\ref{f'/f}) holds.
\end{enumerate}
\end{Thm}
\noindent 
Clearly, the relation (\ref{f'/f})
can also be written as
\[
\bigg|\, \frac{f'(z)}{f(z)} \, \bigg| =
\frac{O(1)}{(1-|z|)^{3\sigma+4+\varepsilon}}
\]
as $|z| \to 1$. But we should note that in case (b) it does not hold
uniformly with respect to $|z|$.

Let $(z_n)_{n \geq 1}$ be a sequence in the unit disc satisfying the
Blaschke condition
\begin{equation} \label{E:blascon}
\sum_{n=1}^{\infty} (1-|z_n|) < \infty.
\end{equation}
Then the Blaschke product
\[
B(z) = \prod_{n=1}^{\infty} \frac{|z_n|}{z_n} \,\,
\frac{z_n-z}{1-\bar{z}_n \, z}
\]
is an analytic function in the unit disc with  order $\sigma=0$ and
\begin{equation} \label{E:estb'11}
\frac{B'(z)}{B(z)}  = \sum_{n=1}^{\infty}
\frac{1-|z_n|^2}{(1-\bar{z}_n \, z)(z-z_n)}.
\end{equation}
Thus Theorem \ref{T:CGH} implies that, for any $\varepsilon>0$,
\[
\bigg|\,\sum_{n=1}^{\infty} \frac{1-|z_n|^2}{(1-\bar{z}_n \,
z)(z-z_n)}\,\bigg| = \frac{O(1)}{(1-|z|)^{4+\varepsilon}}
\]
as $|z| \to 1^-$ in any of the two manners explained above. In this
paper,  instead of (\ref{E:blascon}), we pose more restrictive
conditions on the rate of convergence of zeros $z_n$ and instead we
improve the exponent $4+\varepsilon$. The most common condition is
\begin{equation} \label{E:blasconalalp}
\sum_{n=1}^{\infty} (1-|z_n|)^\alpha < \infty,
\end{equation}
for some $\alpha \in (0,1]$. However, we consider a more general
assumption
\begin{equation} \label{E:blasconalgen}
\sum_{n=1}^{\infty} h(1-|z_n|) < \infty,
\end{equation}
where $h$ is a positive continuous function satisfying certain
smoothness conditions which will be described below.  Our main
prototype for $h$ is
\begin{equation} \label{E:defh}
h(t) = t^\alpha \,\, (\log 1/t)^{\alpha_1} \,\, (\log_2
1/t)^{\alpha_2} \,\, \cdots \,\, \,\, (\log_n 1/t)^{\alpha_n},
\end{equation}
where $\log_n = \log\log\cdots\log$ ($n$ times), $\alpha \in (0,1]$
and $\alpha_1,\alpha_2, \cdots, \alpha_n \in \R$. If $\alpha=1$ the
first nonzero exponent among $\alpha_1,\alpha_2, \cdots, \alpha_n$
is positive \cite{jmcmft}.

The function $h$ is usually defined in an open interval
$(0,\epsilon)$. Of course, by extending its domain of definition, we
may assume that $h$ is defined on the interval $(0,1)$, or if
required, on the entire positive real axis. Moreover, since a
Blaschke sequence satisfies (\ref{E:blascon}), the condition
(\ref{E:blasconalgen}) will provide further information about the
rate of increase of the zeros provided that $h(t) \geq C \, t$ as $t
\to 0$.

The condition (\ref{E:blasconalalp}) has been extensively studied by
many authors \cite{ak74, ak76, ahern79, kutbi02, Li76, Pr73} to
obtain estimates for the integral means of the derivative of
Blaschke products. We \cite{fm1} have recently shown that many of
these estimates can be generalized for Blaschke products satisfying
(\ref{E:blasconalgen}).

\section{Circular Exceptional Sets}
The function $h$ given in (\ref{E:defh}) satisfies the following
conditions:
\begin{enumerate}[a)]
\item $h$ is continuous, positive and increasing with $h(0+)=0$;

\item $h(t)/t$ is decreasing;
\end{enumerate}
In the following, we just need these conditions. Hence, we state our
results for a general function $h$ satisfying $a)$ and $b)$.

\begin{Thm} \label{T:caseles12}
Let $(z_n)_{n \geq 1}$ be a sequence in the unit disc satisfying
\[
\sum_{n=1}^{\infty} h(1-|z_n|) < \infty
\]
and let $B$ be the Blaschke product formed with zeros $z_n$, $n \geq
1$. Let $\beta \geq 1$. Then there is an exceptional set $E \subset
(0,1)$ such that
\[
\int_{E} \,\,\frac{dt}{(1-t)^\beta} \,\, < \infty
\]
and that
\[
\bigg|\, \frac{B'(z)}{B(z)} \,\bigg| =  \frac{o(1)}{(1-|z|)^\beta
\,\, h^2(1-|z|)}
\]
as $|z| \to 1^-$ with $|z| \not\in E$.
\end{Thm}

\begin{proof}
Without loss of generality, assume that $h(t)<1$ for $t \in (0,1)$.
Let
\[
E = \bigcup_{n=1}^{\infty} \bigg(\, |z_n|-(1-|z_n|)^\beta
h(1-|z_n|), \,\, |z_n|+(1-|z_n|)^\beta h(1-|z_n|) \,\bigg).
\]
In the definition of $E$ we implicitly assume that
$|z_n|-(1-|z_n|)^\beta h(1-|z_n|)>0$ in order to have $E \subset
(0,1)$. Certainly this condition holds for large values of $n$. If
it does not hold for some small values of $n$, we simply remove
those intervals from the definition of $E$.

Let $z \in \D$ with $|z| \not\in E$ and fix $0<\delta \leq
(1-|z|)/2$. By (\ref{E:estb'11}), we have
\[
\frac{B'(z)}{B(z)} = \bigg(\, \sum_{\big|\, |z|-|z_n| \,\big| \geq
\delta}+\sum_{\big|\, |z|-|z_n| \,\big| < \delta} \,\bigg) \,
\frac{1-|z_n|^2}{(1-\bar{z}_n \, z)(z-z_n)}.
\]
We use different techniques to estimate each sum. For the first sum
we have
\begin{eqnarray*}
\sum_{\big|\, |z|-|z_n| \,\big| \geq \delta}
\frac{1-|z_n|^2}{|1-\bar{z}_n \, z| \, |z-z_n|}
&\leq&
\frac{2}{\delta} \,\, \sum_{\big|\, |z|-|z_n| \,\big| \geq \delta}
\frac{1-|z_n|}{1-|z_n| \, |z|}.
\end{eqnarray*}
But
\[
\frac{1-|z_n|}{1-|z| \, |z_n|} =
\bigg(\, \frac{1-|z_n|}{h(1-|z_n|)} \, \frac{h(1-|z| \,
|z_n|)}{1-|z| \, |z_n|} \,\bigg) \,\,
\bigg(\,  \frac{h(1-|z_n|)}{h(1-|z| \, |z_n|)} \,\bigg).
\]
Since $h(t)$ is increasing and $h(t)/t$ is decreasing, we get
\[
\frac{1-|z_n|}{1-|z| \, |z_n|} \leq
\frac{h(1-|z_n|)}{h(1-|z|)}
\]
and thus
\[
\sum_{\big|\, |z|-|z_n| \,\big| \geq \delta}
\frac{1-|z_n|^2}{|1-\bar{z}_n \, z| \, |z-z_n|}
\leq
\frac{2\sum_{\big|\, |z|-|z_n| \,\big| \geq \delta}
h(1-|z_n|)}{\delta \,\, h(1-|z|)}
\leq
\frac{C}{\delta \,\, h(1-|z|)}.
\]
A generalized version of this estimation technique has been used in
\cite[Lemma 2.1]{fm1}. To estimate the second sum, we see that
\begin{eqnarray*}
\bigg|\, \frac{1-|z_n|^2}{(1-\bar{z}_n \, z)(z-z_n)} \,\bigg| &\leq&
\frac{2}{|z-z_n|}
\leq \frac{2}{(1-|z_n|)^\beta \, h(1-|z_n|)}\\ %
&\leq&  \frac{C}{(1-|z|)^\beta \, h(1-|z|)},
\end{eqnarray*}
and thus
\[
\bigg|\, \sum_{\big|\, |z|-|z_n| \,\big| < \delta}  \,
\frac{1-|z_n|^2}{(1-\bar{z}_n \, z)(z-z_n)} \,\bigg| \leq \,\, C
\,\,  \frac{n(|z|+\delta)-n(|z|-\delta)}{(1-|z|)^\beta \, h(1-|z|)},
\]
where $n(t)$ is the number of points $z_n$ lying in the disc $\{\, z
\,:\, |z| \leq t \,\}$. Therefore
\begin{equation} \label{E:estim11}
\bigg|\, \frac{B'(z)}{B(z)} \,\bigg| \leq
\frac{C}{h(1-|z|)} \,\,\bigg(\, \frac{1}{\delta} +
\frac{n(|z|+\delta)-n(|z|-\delta)}{(1-|z|)^\beta} \,\bigg)
\end{equation}
provided that $z \in \D$ with $|z| \not\in E$. The best choice of
$\delta$ depends on the counting function $n(t)$. We make a choice
for the most general case.

Assume that $\delta = (1-|z|)/2$. Our assumption
(\ref{E:blasconalgen}) on the rate of increase of zeros $z_n$ is
equivalent to
\[
\int_{0}^{1} \, h(1-t) \,\, dn(t) < \infty,
\]
and it is well known that this condition implies
\begin{equation} \label{E:s2Cin22}
n(t) = \frac{o(1)}{h(1-t)}
\end{equation}
as $t \to 1^-$. Therefore,
\begin{equation} \label{E:s2C22}
n(|z|+\delta)-n(|z|-\delta) \leq \frac{o(1)}{h(1-|z|)}.
\end{equation}
Hence, by (\ref{E:estim11}) and (\ref{E:s2C22}), we get the promised
growth for $B'/B$. To verify the size of $E$, note that
\begin{eqnarray*}
\int_{E} \,\, \frac{dt}{(1-t)^\beta}
&=& \sum_{n=1}^{\infty} \int_{|z_n|-(1-|z_n|)^\beta
h(1-|z_n|)}^{|z_n|+(1-|z_n|)^\beta h(1-|z_n|)}
\,\,\, \frac{dt}{(1-t)^\beta} \\
&=& \sum_{n=1}^{\infty} \int_{(1-|z_n|)-(1-|z_n|)^\beta
h(1-|z_n|)}^{(1-|z_n|)+ (1-|z_n|)^\beta h(1-|z_n|)}
\,\,\, \frac{d\tau}{\tau^\beta}\\
&\leq&  \,\, \sum_{n=1}^{\infty} \frac{2(1-|z_n|)^\beta h(1-|z_n|)}{(\,\, (1-|z_n|)- (1-|z_n|)^\beta h(1-|z_n|) \,\,)^\beta}\\
&\leq&   C \,\, \sum_{n=1}^{\infty} h(1-|z_n|) < \infty.
\end{eqnarray*}

\end{proof}

{\em Remark 1:} As the counting function $n(t) = 1/(1-t)^\alpha$
suggests, the assumption
\begin{equation} \label{E:s2C222}
n(|z|+\delta) - n(|z|-\delta) \leq C \,\, \frac{\delta \,
n(|z|)}{1-|z|}
\end{equation}
is fulfilled by a wide class of distribution of zeros. If
(\ref{E:s2C222}) holds, by (\ref{E:s2C22}) and (\ref{E:estim11})
with
\[
\delta = (1-|z|)^{\frac{1+\beta}{2}} \, h^{\frac{1}{2}}(1-|z|),
\]
we obtain
\[
\bigg|\, \frac{B'(z)}{B(z)} \,\bigg| =
\frac{O(1)}{(1-|z|)^{\frac{1+\beta}{2}} \, h^{\frac{3}{2}}(1-|z|)}
\]
as $|z| \to 1^-$ with $|z| \not\in E$.

{\em Remark 2:} Let us call $\varphi$ almost increasing if
$\varphi(x) \leq \mbox{Const} \, \varphi(y)$ provided that $x \leq
y$. Almost decreasing functions are defined similarly. As it can be
easily verified, Theorem \ref{T:caseles12} (and also Theorem
\ref{T:caseles13}) is still true if we assume that $h(t)$ is almost
increasing and $h(t)/t$ is almost decreasing.

\begin{Cor} \label{C:caseles12}
Let $\alpha \in (0,1]$, and $\alpha_1,\alpha_2,\cdots,\alpha_n \in
\R$.  Let  $(z_n)_{n \geq 1}$ be a sequence in the unit disc with
\[
\sum_{n=1}^{\infty} (1-|z_n|)^\alpha \,\, (\log
1/(1-|z_n|))^{\alpha_1} \,\,  \cdots \,\, \,\, (\log_n
1/(1-|z_n|))^{\alpha_n} < \infty
\]
and let $B$ be the Blaschke product formed with zeros $z_n$, $n \geq
1$. Let $\beta \geq 1$. Then there is an exceptional set $E \subset
(0,1)$ such that
\[
\int_{E} \,\, \frac{dt}{(1-t)^\beta} \,\, < \infty
\]
and that
\begin{equation} \label{E:b'balp}
\bigg|\, \frac{B'(z)}{B(z)} \,\bigg| =
\frac{o(1)}{(1-|z|)^{\beta+2\alpha} \, (\log 1/(1-|z|))^{2\alpha_1}
\,\,
 \cdots \,\, \,\, (\log_n
1/(1-|z|))^{2\alpha_n}}
\end{equation}
as $|z| \to 1^-$ with $|z| \not\in E$.
\end{Cor}
\noindent In particular, if
\begin{equation} \label{E:alpha222}
\sum_{n=1}^{\infty} (1-|z_n|)^\alpha < \infty,
\end{equation}
then, for any $\beta \geq 1$, there is an exceptional set $E \subset
(0,1)$ such that
\begin{equation} \label{E:E222}
\int_{E} \,\, \frac{dt}{(1-t)^\beta}  \,\, < \infty
\end{equation}
and that
\[
\bigg|\, \frac{B'(z)}{B(z)} \,\bigg| =
\frac{o(1)}{\,\,(1-|z|)^{\beta+2\alpha}}
\]
as $|z| \to 1^-$ with $|z| \not\in E$. If $(|z_n|)_{n \geq 1}$ is an
interpolating sequence then
\[
1-|z_{n+1}| \leq c \, (1-|z_n|)
\]
for a constant $c<1$ \cite[Theorem 9.2]{Co83}. Hence,
(\ref{E:alpha222}) is satisfied for any $\alpha>0$ and thus, for any
$\beta \geq 1$ and for any $\varepsilon>0$, there is an exceptional
set $E$ satisfying (\ref{E:E222}) such that
\begin{equation} \label{E:E2232}
\bigg|\, \frac{B'(z)}{B(z)} \,\bigg| =
\frac{o(1)}{\,\,(1-|z|)^{\beta+\varepsilon}}
\end{equation}
as $|z| \to 1^-$ with $|z| \not\in E$. It is interesting to know if
in  (\ref{E:E2232}) we are able to replace $\varepsilon$ by zero.

\section{Radial Exceptional Sets}
Contrary to the preceding section, we now study the behavior of
\[
\bigg|\, \frac{B'(re^{i\theta})}{B(re^{i\theta})} \,\bigg|
\]
as $r \to 1$ for a {\em fixed} $\theta$. We obtain an upper bound
for the quotient $B'/B$ as long as $e^{i\theta} \in \T \setminus E$
where $E$ is an exceptional set of Lebesgue measure zero.

\begin{Thm} \label{T:caseles13}
Let $B$ be the Blaschke product formed with zeros $z_n=r_n
e^{i\theta_n}$, $n \geq 1$, satisfying
\[
\sum_{n=1}^{\infty} h(1-r_n) < \infty.
\]
Then there is an exceptional set $E \subset \T$ whose Lebesgue
measure $|E|$  is zero such that for all $z=re^{i\theta}$ with
$e^{i\theta} \in \T \setminus E$
\[
\bigg|\, \frac{B'(z)}{B(z)} \,\bigg| =  \frac{o(1)}{(1-|z|) \,\,
h(1-|z|)}
\]
as $|z| \to 1^-$.
\end{Thm}

\begin{proof}
Let us consider the open set
\[
U_n = \{\, z \in \D \,:\, (1-|z|) > C|z-z_n| \,\}
\]
with $C>1$, and we define
\[
I_n =\{\, \zeta \in \T \,:\, \exists z \in U_n \,\, \&\, \zeta=z/|z|
\,\}.
\]
In other words, $I_n$ is the radial projection of $U_n$ on the unit
circle $\T$. Then we know that
\begin{equation} \label{E:sizein}
|I_n| \leq C' (1-r_n),
\end{equation}
where $C'$ is a constant just depending on $C$. Let
\[
E = \bigcap_{n=1}^{\infty} \bigcup_{k=n}^{\infty} I_k.
\]
By (\ref{E:sizein}), we see that $|E|=0$.

Fix $z \in \D$ with $z/|z| \not\in E$. Hence, there is $N$ such that
$z/|z| \not\in I_k$ for all $k \geq N$. Let $R = (1+|z|)/2$. Now, we
write
\[
\frac{B'(z)}{B(z)} = \bigg(\, \sum_{|z_n| \geq R}+\sum_{|z_n| <R, \,
n \geq N}+\sum_{n=1}^{N-1} \,\bigg) \, \frac{1-|z_n|^2}{(1-\bar{z}_n
\, z)(z-z_n)},
\]
and as in the preceding case
\begin{equation} \label{E:estim1=1}
\sum_{|z_n| \geq R} \frac{1-|z_n|^2}{|1-\bar{z}_n \, z| \, |z-z_n|}
\leq
\frac{o(1)}{(1-|z|) \,\, h(1-|z|)}.
\end{equation}

\noindent To estimate the second sum, we see that
\[
\bigg|\, \frac{1-|z_n|^2}{(1-\bar{z}_n \, z)(z-z_n)} \,\bigg| \leq
\frac{2}{|z-z_n|}
\leq \frac{2C}{1-|z|},  \hspace{1cm} (|z| \not\in E),
\]
and thus, by (\ref{E:s2Cin22}),
\begin{equation} \label{E:s2C=1}
\bigg|\, \sum_{|z_n| <R, \, n \geq N}  \,
\frac{1-|z_n|^2}{(1-\bar{z}_n \, z)(z-z_n)} \,\bigg| \leq \frac{2C
\, n(R)}{1-|z|}\leq \frac{o(1)}{(1-|z|) \, h(1-|z|)}.
\end{equation}
Since the last sum is uniformly bounded ($\theta$ is fixed),
(\ref{E:estim1=1}) and (\ref{E:s2C=1}) give the required result.
\end{proof}

\begin{Cor} \label{C:caseles13}
Let $\alpha \in (0,1]$, and $\alpha_1,\alpha_2,\cdots,\alpha_n \in
\R$. If $\alpha=1$ the first nonzero number among
$\alpha_1,\alpha_2, \cdots, \alpha_n$ is positive. Let $B$ be the
Blaschke product formed with zeros $z_n=r_n e^{i\theta_n}$, $n \geq
1$, satisfying
\[
\sum_{n=1}^{\infty} (1-r_n)^\alpha \,\, (\log 1/(1-r_n))^{\alpha_1}
\,\,  \cdots \,\, \,\, (\log_n 1/(1-r_n))^{\alpha_n} < \infty.
\]
Then there is an exceptional set $E \subset \T$ whose Lebesgue
measure $|E|$  is zero such that for all $z=re^{i\theta}$ with
$e^{i\theta} \in \T \setminus E$
\begin{equation} \label{E:b'balp}
\bigg|\, \frac{B'(z)}{B(z)} \,\bigg| =
\frac{o(1)}{(1-|z|)^{1+\alpha} \, (\log 1/(1-|z|))^{\alpha_1} \,\,
 \cdots \,\, \,\, (\log_n
1/(1-|z|))^{\alpha_n}}
\end{equation}
as $|z| \to 1^-$.
\end{Cor}
\noindent In particular, if
\[
\sum_{n=1}^{\infty} (1-r_n)^\alpha < \infty,
\]
then there is an exceptional set $E \subset \T$ whose Lebesgue
measure $|E|$  is zero such that for all $z=re^{i\theta}$ with
$e^{i\theta} \in \T \setminus E$
\begin{equation} \label{E:b'balp}
\bigg|\, \frac{B'(z)}{B(z)} \,\bigg| =
\frac{o(1)}{(1-|z|)^{1+\alpha}}
\end{equation}
as $|z| \to 1^-$.

\noindent {\bf Remark:} Theorems \ref{T:caseles12} and
\ref{T:caseles13} can be easily generalized to obtain estimates for
\[
\frac{B^{(k)}(z)}{B^{(j)}(z)}
\]
as $|z| \to 1^-$. This is a standard technique which can been find
for example in \cite{kutbi02, Li76}. 

\end{document}